\documentclass[11pt,a4paper]{amsart}
\usepackage{epsfig}
\usepackage{amsmath}
\usepackage{amsthm}
\usepackage{amssymb}
\usepackage{bm}
\usepackage{tikz}
\usepackage{PTSansNarrow}
\usepackage[T1]{fontenc}
\usetikzlibrary{matrix}
\usepackage{graphicx}
 \numberwithin{dummy}{section}

\newtheorem*{algorithm-PF}{Primal Formulation}
\newtheorem*{algorithm-PMF}{Primal-Mixed Formulation}
\newtheorem*{algorithm-DMF}{Dual-Mixed Formulation}
\newtheorem*{algorithm-PF-C}{Conforming Finite Element Method}
\newtheorem*{algorithm-PMF-C}{Conforming Primal-Mixed Finite Element Method}
\newtheorem*{algorithm-DMF-C}{Conforming Mixed Finite Element Method}
\newtheorem*{algorithm-hmfem}{Hybridized Mixed Finite Element Method}
\newtheorem*{algorithm-primal-wgfem}{Primal WG-FEM}
\newtheorem*{algorithm-primalmixed-wgfem}{Primal-Mixed WG-FEM}
\newtheorem*{algorithm-mixed-wgfem}{Mixed WG-FEM}
\newtheorem*{algorithm-hmwgfem}{Hybridized Mixed WG-FEM}
\newtheorem*{algorithm-hdg}{Hybridizable Discontinuous Galerkin}
\newtheorem{remark}{Remark}
\numberwithin{equation}{section}
\newcommand{\bq}{{\bf q}}
\newcommand{\bn}{{\bf n}}
\newcommand{\bx}{{\bf x}}

\newcommand{\bv}{{\bf v}}
\newcommand{\bV}{{\bf V}}

\def\T{{\mathcal T}}
\def\E{{\mathcal E}}

\def\l{{\langle}}
\def\r{{\rangle}}

\def\bn{{\bf n}}
\def\bq{{\bf q}}

\def\bg{{\bm\gamma}}

\def\bbQ{\mathbb{Q}}

\newcommand{\pT}{{\partial T}}

\newtheorem{proposition}{Proposition}[section]
\def\3bar{{|\hspace{-.02in}|\hspace{-.02in}|}}
\renewcommand{\ldots}{\dotsc}
\begin{document}
\title{The basics of Weak Galerkin finite element methods}

\author{Junping Wang}
\email{junpingwang@gmail.com}
\author{Xiu Ye}
\address{Department of
Mathematics, University of Arkansas at Little Rock, Little Rock, AR
72204}
\email{xxye@ualr.edu}


\begin{abstract}
The goal of this article is to clarify some misunderstandings and inappropriate claims made in \cite{cock} regarding the relation between the weak Galerkin (WG) finite element method and the hybridizable discontinuous Galerkin (HDG). In this paper, the authors offered their understandings and interpretations on the weak Galerkin finite element method by describing the basics of the WG method and how WG can be applied to a model PDE problem in various variational forms. In the authors' view, WG-FEM and HDG methods are based on different philosophies and therefore represent different methodologies in numerical PDEs, though they share something in common in their roots. A theory and an example are given to show that the primal WG-FEM is not equivalent to the existing HDG \cite{cgl}.
\end{abstract}



\pagestyle{myheadings}

\maketitle

\section{Introduction}\label{Section:Introduction}

In this article we shall discuss the basics for the weak Galerkin finite element methods when applied to the following model problem: Find an unknown function $u$ from appropriate spaces such that
\begin{eqnarray}
-\nabla\cdot (a\nabla u)&=&f,\quad \mbox{in}\;\Omega,\label{pde}\\
u&=&0,\quad\mbox{on}\;\partial\Omega,\label{bc}
\end{eqnarray}
where $\Omega$ is a polygonal or polyhedral domain in $\mathbb{R}^d\ (d=2,3)$. The coefficient tensor $a=a(\bx)$ is assumed to be bounded, piecewise continuous, and uniformly symmetric and positive-definite in the domain $\Omega$. The article is written with the goal of clarifying some misunderstandings or inappropriate claims made in \cite{cock} regarding the relation between the weak Galerkin (WG) finite element method and the hybridizable discontinuous Galerkin (HDG) method. In this paper, we would like to offer our understandings and interpretations on the WG finite element method by describing the basics of the WG method and show how WG method can be applied to PDE problems. The model problem \eqref{pde}-\eqref{bc} is chosen merely for simplicity of the presentation.

The development of the WG finite element methods has been a learning and discovering process to the authors in the last eight years. The original idea of weak Galerkin stems out of the concept of ``discrete weak gradient'' or ``weak gradient'' presented in \cite{wg} and a workshop talk by one of the authors in the Chern Institute of Mathematics at Nankai University, China in June 2011. The concept of weak gradient was inspired by the development of exterior calculus that has been an active research topic in the last two decades. The use of stabilizer/smoother in the weak Galerkin formulations should be attributed to the development of the stabilized finite element method for the Stokes equation \cite{franca,JimJunping}, which leads to the improved and generalized versions of the weak Galerkin finite element methods (see, for example, \cite{mixed,wg-soe,wg-PkPk-1,wg-systematic,ww2016,ww2017,ww2018,
wg-elasticity,wg-biharmonic,CJ-biharmonic}). The authors have been gaining more and improved understandings on the basic principles of the weak Galerkin method and its potential in numerical PDEs since 2011 through the development of new WG schemes with the help from their collaborators and many other researchers in the scientific community. It has been demonstrated by recent publications that the WG methods, including the primal-dual weak Galerkin finite element method developed in \cite{ww2016,ww2017,ww2018}, are applicable to a wide range of PDE problems and have very promising potentials in scientific computing.
With respect to connections to HDG, in the authors' view, WG and HDG methods are based on different philosophies and therefore represent different methodologies in numerical PDEs, though they share something in common in their roots.

\medskip
The following is an overview of the content of the paper. For simplicity, we adopt the abbreviation FEM for {\em Finite Element Method}.
\begin{enumerate}
\item {\color{blue}\em Three variational formulations:} Section \ref{Section:3forms} is devoted to a presentation of three variational formulations for the model problem \eqref{pde}-\eqref{bc} \cite{robert-thomas}, which are: (1) primal variational form, (2) primal-mixed variational form, and (3) dual-mixed variational form (or, simply, mixed variational form).
\item {\color{blue} \em Three conforming FEMs:} For each of the variational forms, we shall discuss the corresponding conforming finite element method by following the usual Galerkin method; see Section \ref{Section:CFEMs} for details.
\item {\color{blue} \em Hybridized Mixed FEM:} For the conforming mixed finite element method (one of the three conforming FEMs), we shall describe its hybridized formulation in Section \ref{Section:hmfem} by following the approach originated in Fraeijs de Veubeke \cite{hmf}.
\item {\color{blue} \em Basic Principles of Weak Galerkin:} In Section \ref{Section:WG}, we shall introduce weak Galerkin as a generic numerical methodology for PDEs by describing its basic principles and characteristics. For each of the three variational formulations, we shall develop the corresponding weak Galerkin FEMs, which are:
    \begin{itemize}
    \item {\color{blue} Primal WG-FEM:} primal weak Galerkin finite element method,
     \item {\color{blue} Primal-Mixed WG-FEM:} primal-mixed weak Galerkin finite element method,
     \item {\color{blue} Mixed WG-FEM:} dual-mixed weak Galerkin finite element method, or simply, mixed weak Galerkin finite element method.
    \end{itemize}
\item {\color{blue}\em Hybridized Mixed WG-FEM:} In Section \ref{Section:hmwgfem}, we shall present a hybridized formulation for the Mixed WG-FEM by following the method of Fraeijs de Veubeke in \cite{hmf}. The {\em Hybridized Mixed WG-FEM} should be viewed as a parallel development of the {\em Hybridized Mixed FEM} in the weak Galerkin context.
\item {\color{blue} \em Reformulation of the Hybridized Mixed WG-FEM:} This reformulation of the {\em Hybridized Mixed WG-FEM} will be used for making comparisons between the {\em Hybridized Mixed WG-FEM} and the {\em Hybridizable Discontinuous Galerkin (HDG) Method} \cite{cock,cgl}, see Section \ref{Section:hmwgfem-re} for details.
\item {\color{blue}\em Comparison between HDG and the Hybridized Mixed WG-FEM:} In Section \ref{Section:hdg}, we shall investigate the similarities and the differences on the unknown variables and the governing linear equations between the HDG and the {\em Hybridized Mixed WG} methods. In particular, we demonstrate that HDG is a special case of the {\em Hybridized Mixed WG} when the finite element spaces are constructed in special ways.
\item {\color{blue} \em Primal WG $\neq$ HDG:} In Section \ref{Section:WG-NE-HDG}, we show that the primal WG-FEM is different from the existing HDG. The justification is given by a theory through reformulations for both the HDG and the weak Galerkin methods into forms that use the same set of basis functions, assuming that the same finite element spaces are employed in the discretizations. Furthermore, an illustrative example is included in this section to show the clear difference between the two methods.
\item {\color{blue} \em Remarks to \cite{cock}:} In Section \ref{Section:Remarks}, we shall make several remarks regarding a second HDG formulation and some claims made in \cite{cock}.
\end{enumerate}

Throughout the paper, we follow the usual notation for Sobolev spaces and norms. For any open bounded domain $D\subset \mathbb{R}^d$ ($d$-dimensional Euclidean space) with Lipschitz continuous boundary, we use $\|\cdot\|_{s,D}$ and $|\cdot|_{s,D}$ to denote the norm and seminorm in the Sobolev space $H^s(D)$ for any $s\ge 0$, respectively. The inner product in
$H^s(D)$ is denoted by $(\cdot,\cdot)_{s,D}$. The space $H^0(D)$ coincides with $L^2(D)$, for which the norm and the inner product are denoted by $\|\cdot \|_{D}$ and $(\cdot,\cdot)_{D}$, respectively. When $D=\Omega$, we shall drop the subscript $D$ in the norm and inner product notation. $H_0^1(\Omega)$ is the closed subspace of $H^1(\Omega)$ consisting of functions with vanishing trace on the boundary $\partial\Omega$. The Sobolev space $H(div;\Omega)$ consists of $L^2(\Omega)$ vector-valued functions with square integrable divergence.

\section{Three Variational Formulations}\label{Section:3forms}

In literature (cf. \cite{robert-thomas} for example), there are at least three variational formulations developed for the model second order elliptic problem (\ref{pde})-(\ref{bc}) to define weak solutions with various characteristics. Each of the variational formulations gives rise to a particular class of conforming finite element methods through the use of the standard Galerkin method formulated in abstract Hilbert spaces. This section shall describe three well-known variational formulations which form the starting point for three conforming Galerkin finite element methods and furthermore for sparking the development of weak Galerkin finite element methods.

The first variational formulation is known as the primal formulation -- the most commonly used one for the model problem \eqref{pde}-\eqref{bc}. The primal variational form can be obtained by testing the equation \eqref{pde} against any $v\in H_0^1(\Omega)$ with a use of the usual integration by parts for the model problem \eqref{pde}-\eqref{bc}. The following is a precise statement of the primal variational form:

\begin{algorithm-PF}
The primal variational form for the boundary
value problem (\ref{pde})-(\ref{bc}) is to find $u\in H_0^1(\Omega)$
such that
\begin{equation}\label{pf}
(a\nabla u,\nabla v)=(f,v)\qquad \forall v\in H_0^1(\Omega).
\end{equation}
\end{algorithm-PF}

For presenting the other two variational formulations, we shall rewrite the equation (\ref{pde}) into a system of first order partial differential equations by introducing the flux variable $\bq=-a\nabla u$ so that (\ref{pde})-(\ref{bc}) can be reformulated as follows:
\begin{eqnarray}
a^{-1}\bq+\nabla u&=&{\mathbf 0},\qquad \mbox{in } \Omega,\label{mix:001}\\
\nabla\cdot \bq&=&f,\qquad \mbox{in }\Omega,\label{mix:002}\\
u&=&0,\qquad\mbox{on } \partial\Omega.\label{mix:003}
\end{eqnarray}

By first testing the equation \eqref{mix:001} against any $\bv\in [L^2(\Omega)]^d$ and then the equation \eqref{mix:002} against any $w\in H_0^1(\Omega)$ with integration by parts,  we arrive at the following primal-mixed variation problem for \eqref{pde}-\eqref{bc}:

\begin{algorithm-PMF} The primal-mixed variational form for the  model problem (\ref{pde})-(\ref{bc}) is to seek $\bq\in  [L^2(\Omega)]^d$ and $u\in H_0^1(\Omega)$ such that
\begin{eqnarray}
(a^{-1}\bq,\bv)+(\bv,\nabla u)&=&0
\qquad\forall\bv\in [L^2(\Omega)]^d,\label{pm1}\\
(\bq,\nabla w)&=&-(f,w)\quad\forall w\in H_0^1(\Omega).\label{pm2}
\end{eqnarray}
\end{algorithm-PMF}

\medskip

Next, we test the equation \eqref{mix:001} against any $\bv\in H(div;\Omega)$ and apply the usual integration by parts $(\bv,\nabla u ) = - (\nabla\cdot\bv,u)$ to obtain the following dual-mixed variational form:

\begin{algorithm-DMF} The dual-mixed variational form for the model problem (\ref{pde})-(\ref{bc}) is to seek  $\bq\in  H(div,\Omega)$ and $u\in L^2(\Omega)$ such that
\begin{eqnarray}
(a^{-1}\bq,\bv)-(\nabla\cdot\bv,u)&=&0
\qquad\forall\bv\in  H(div,\Omega),\label{dm1}\\
(\nabla\cdot\bq,w)&=&(f,w)\qquad\forall w\in L^2(\Omega).\label{dm2}
\end{eqnarray}
\end{algorithm-DMF}
\smallskip

\section{Conforming Galerkin finite element methods}\label{Section:CFEMs}

Given a finite element triangulation $\T_h$ of the polygonal or polyhedral domain $\Omega$ by $d$-simplexes $T\in\T_h$ and a positive integer $k>0$, we define a finite dimensional space $V_h:=V_h^{(k)}$ as follows:
\begin{equation}\label{Vh}
V_h^{(k)}=\{v\in H_0^1(\Omega): \  v|_T\in P_k(T)\;\; \forall T\in\T_h\}.
\end{equation}
$V_h$ is clearly a subspace of $H_0^1(\Omega)$ consisting of continuous piecewise polynomials. The usual Galerkin method with the use of $V_h\subset H_0^1(\Omega)$ and the primal variational formulation \eqref{pf} leads to the following conforming Galerkin finite element method:
\begin{algorithm-PF-C} Let $V_h$ be given by \eqref{Vh}. The conforming Galerkin finite element method seeks $u_h\in V_h$ satisfying
\begin{equation}\label{pf-c}
(a\nabla u_h,\nabla v)=(f,v)\qquad \forall v\in V_h.
\end{equation}
\end{algorithm-PF-C}

\medskip
Next, we introduce another finite dimensional space $W_h:=W_h^{(k)}$ as follows:
\begin{equation}\label{Wh}
W_h^{(k)}=\left\{\bq\in [L^2(\Omega)]^d: \  \bq|_T\in [P_{k-1}(T)]^d\;\; \forall T\in\T_h \right\}.
\end{equation}
It is clear that $W_h$ is a subspace of $[L^2(\Omega)]^d$. The usual Galerkin method with the use of $W_h\times V_h\subset [L^2(\Omega)]^d\times H_0^1(\Omega)$ and the primal-mixed formulation \eqref{pm1}-\eqref{pm2} gives rise to the following conforming primal-mixed finite element method for the model problem \eqref{pde}-\eqref{bc}(cf. \cite{reddy-oden, babuska-oden-lee}):

\begin{algorithm-PMF-C} The primal-mixed finite element method for the  model problem (\ref{pde})-(\ref{bc}) is to seek $\bq_h\in  W_h$ and $u_h\in V_h$ such that
\begin{eqnarray}
(a^{-1}\bq_h,\bv)+(\bv,\nabla u_h)&=&0
\qquad\forall\bv\in W_h,\label{pm1-c}\\
(\bq_h,\nabla w)&=&-(f,w)\quad\forall w\in V_h.\label{pm2-c}
\end{eqnarray}
\end{algorithm-PMF-C}

\medskip

To construct a conforming dual-mixed approximation for the solution of \eqref{dm1}-\eqref{dm2}, we introduce a finite dimensional subspace $Z_h:=Z_h^{(k)}$ of $H(div;\Omega)$ and a finite dimensional subspace $U_h:=U_h^{(k)}$ of $L^2(\Omega)$ given as follows:
\begin{equation}\label{Wh-mixed-vector}
Z_h^{(k)}=\left\{\bq\in H(div;\Omega): \ \bq|_T\in D_k(T)\; \; \forall T\in\T_h\right\},
\end{equation}
and
\begin{equation}\label{Wh-mixed-scalar}
U_h^{(k)}=\{w\in L^2(\Omega): \  w|_T\in P_{k-1}(T)\; \; \forall T\in\T_h\},
\end{equation}
where $D_k(T):=[P_{k-1}(T)]^d+\bx P_{k-1}(T)$ is the space of vector-valued polynomials on $T$ in the form of $\bq=(q_1, q_2,\ldots, q_d)+(x_1, x_2,\ldots, x_d) q_0$ for some $q_i\in P_{k-1}(T),\ i=0,1,\ldots,d$. The usual Galerkin method based on the dual-mixed variational formulation \eqref{dm1}-\eqref{dm2} and the subspaces $Z_h\times U_h$ can be described as follows:

\begin{algorithm-DMF-C} The mixed finite element method (also known as the {\bf dual-mixed finite element method} in literature \cite{robert-thomas}) for the model problem (\ref{pde})-(\ref{bc}) is to seek $\bq_h\in Z_h $ and $u_h\in U_h$ satisfying
\begin{eqnarray}
(a^{-1}\bq_h,\bv)-(\nabla\cdot\bv,u_h)&=&0
\qquad\forall\bv\in  Z_h,\label{dm1-discrete}\\
(\nabla\cdot\bq_h,w)&=&(f,w)\qquad\forall w\in U_h.\label{dm2-discrete}
\end{eqnarray}
\end{algorithm-DMF-C}

The spaces $Z_h\times U_h$ were proposed and analyzed in \cite{rt} in the two-dimensional case (i.e., $d=2$) and in \cite{nedelec} in the three-dimensional case. In the literature, the space $Z_h\times U_h$ or $Z_h^{(k)}\times U_h^{(k)}$ is referred to as the Raviart-Thomas space of index $k$ (or sometimes $k-1$) or, in the three-dimensional case, as the Raviart-Thomas-Nedelec space. Other examples of the finite element spaces $Z_h\times U_h$ include the Brezzi-Douglas-Marini element \cite{bdm}; see \cite{mixed} for more examples.

\begin{remark}
The three conforming finite element methods (namely, the conforming finite element method, the conforming primal-mixed finite element method, and the conforming mixed finite element method) represent three distinct classes of finite element methods for the model problem \eqref{pde}-\eqref{bc}. They are clearly different numerical schemes for the general second order elliptic equations illustrated by the model problem \eqref{pde}-\eqref{bc}. The principle of the conforming finite element methods is applicable to a wide class of partial differential equations, which has been a major research direction in computational mathematics and scientific computing in the last several decades.
\end{remark}

\section{Hybridized mixed finite element method}\label{Section:hmfem}

The mixed finite element method (or dual-mixed finite element method) \eqref{dm1-discrete}-\eqref{dm2-discrete} can be hybridized by following the procedure originally developed in Fraeijs de Veubeke \cite{hmf}. Observe that the finite element functions $\bv\in Z_h\subset H(div;\Omega)$ must be continuous along the normal direction across each inter-element interface $\pT_1\cap\pT_2$ for any $T_i\in\T_h,\ i=1,2$. The idea of Fraeijs de Veubeke hybridization is to eliminate the inter-element continuity requirements from the space $Z_h$ thereby obtaining a space $\tilde Z_h$ and to impose instead the desired continuity on the solution $\bq_h\in \tilde Z_h$ via Lagrangian multipliers.

Denote by $\partial\T_h$ the set of the element boundaries of $\T_h$; i.e.,
$\partial\T_h=\bigcup_{T\in\T_h} \partial T$. For the given triangulation $\T_h$, we introduce the following finite element space:
\begin{equation}\label{Wh-mixed-vector-hybridized}
\tilde Z_h^{(k)}=\prod_{T\in\T_h}\left\{\bq\in [L^2(T)]^d: \  \bq|_T\in D_k(T)\right\}.
\end{equation}
Moreover, we denote by $\Lambda_h\subset L^2(\partial\T_h)$ the finite dimensional space defined by
\begin{equation}\label{LagrangeM-Space}
\Lambda_h=\left\{\sigma: \ \forall T\in\T_h, \exists\ \bq\in D_k(T)\ s.t.\ \bq|_\pT\cdot\bn_T=\sigma|_\pT \right\},
\end{equation}
where $\bn_T$ stands for the outward normal direction on $\pT$. Basically, $\Lambda_h$ consists of piecewise polynomials of degree $k-1$ on the set of element boundaries $\partial\T_h$. Denote by $\Lambda_h^0\subset \Lambda_h$ the subspace of $\Lambda_h$ with vanishing value on $\partial\Omega$; i.e.,
$$
\Lambda_h^0=\left\{\sigma\in\Lambda_h: \ \ \sigma|_{\pT\cap\partial\Omega}=0 \ \ \forall T\in\T_h\right\}.
$$
$\Lambda_h^0$ is best known as the Lagrangian space in the literature of mixed finite element method.

Throughout the paper, we use $\langle \phi, \psi\rangle_\pT$ to denote the $L^2(\pT)$ inner product for any $\phi, \psi\in L^2(\pT)$.

\begin{algorithm-hmfem}
Find $\bq_h\in \tilde{Z}_h$, $u_h\in U_h$, and $u_b\in \Lambda_h^0$ satisfying
\begin{eqnarray}
(a^{-1}\bq_{h}, \bv)_T - (\nabla\cdot\bv, u_h)_T +\langle u_b, \bv\cdot\bn_T\rangle_\pT &=& 0\quad \forall T\in\T_h,\ \bv\in D_k(T),\label{hmfem:01}\\
(\nabla\cdot\bq_h, w)_T &=& (f, w)_T\qquad \forall T\in\T_h,\ w\in U_h,\label{hmfem:02}\\
\sum_{T\in\T_h}\langle \sigma, \bq_h|_{\pT}\cdot\bn_T\rangle_\pT &=& 0\qquad \forall \sigma\in \Lambda_h^0.\label{hmfem:03}
\end{eqnarray}
\end{algorithm-hmfem}

The hybridized mixed finite element method \eqref{hmfem:01}-\eqref{hmfem:03} is equivalent to the mixed finite element method \eqref{dm1-discrete}-\eqref{dm2-discrete} in the sense that the numerical solutions for $\bq_h$ and $u_h$ resulting from both schemes are identical. In fact, the equation \eqref{hmfem:03} implies that the solution $\bq_h$ from \eqref{hmfem:01}-\eqref{hmfem:03} has continuous normal component across each inter-element interface. Thus, by restricting the test function $\bv$ to the closed subspace $Z_h\subset \tilde Z_h$ in \eqref{hmfem:01} one arrives at the system of equations \eqref{dm1-discrete}-\eqref{dm2-discrete}. Finally, the uniqueness of the solutions for \eqref{dm1-discrete}-\eqref{dm2-discrete} leads to the desired conclusion.

\section{Weak Galerkin Finite Element Methods}\label{Section:WG}

To the authors' understanding, \underline{\em Weak Galerkin Finite Element Method (WG-FEM)} is a generic numerical/discretization methodology for partial differential equations with the following guiding principles/characteristics:
\begin{enumerate}
\item {\bf\color{blue} Variational or weak form based:} the numerical method is based on a variational or weak form formulation for the underlying PDE problems;
\item {\bf\color{blue} Subspace approximation through domain partitioning:} the computational or physical domain where the PDEs are defined is partitioned into small subdomains (also known as {\em elements}), and the approximating functions are constructed on each subdomain or element;
\item {\bf\color{blue} Discrete weak differential operators:} the differential operators that are used to define the variational or weak form are locally reconstructed on each element using problem-independent tools. This could form a set of building blocks which may constitute a WG calculus;
\item {\bf\color{blue} Weak regularity/smoothness:} the continuity or regularity necessary to define the functions in the corresponding Sobolev (or Banach) spaces with which the variational or weak forms are defined is characterized by using carefully chosen stabilizers/smoothers for the approximating functions.
\end{enumerate}

In the rest of this section we shall apply the above principles to the model problem \eqref{pde}-\eqref{bc}. In particular, the three variational formulations described in Section \ref{Section:3forms} shall be considered; each will yield a particular class of weak Galerkin finite element method.

For simplicity, assume that the domain $\Omega$ is polygonal (for 2D problems) or polyhedral (for 3D problems). Let $\T_h$ be a polygonal or polyhedral partition of $\Omega$ consisting of general polygons or polyhedra. The set of element boundaries is denoted as $\partial\T_h=\bigcup_{T\in\T_h} \partial T$. Denote by $\E_h=\{e\}$ a finite element partition of $\partial\T_h$ consisting of elements of dimension $d-1$. Assume that the partition $\E_h$ is consistent with $\T_h$ in the following sense:
\begin{itemize}
\item For any $T\in \T_h$, the boundary $\pT$ is the collection of some elements in $\E_h$; i.e., $\pT=\bigcup_{i=1}^{N_T} e_i$ for some $e_i\in\E_h$;
\item For any $T_1,T_2\in\T_h$ such that $dim(\pT_1\cap\pT_2)=d-1$, the interface $\pT_1\cap\pT_2$ is the collection of some elements in $\E_h$; i.e., $\pT_1\cap\pT_2 = \bigcup_{j=1}^{N_{T_1T_2}} e_j$ for some $e_j\in\E_h$.
\end{itemize}

Observe that the polygonal or polyhedral partition $\T_h$ enjoys more flexibility in its construction than those used in the conforming finite element method in that: (1) each element can be of arbitrary shape, and (2) the element interface $\pT_1\cap\pT_2$ can be a portion of a flat side/face of $\pT_i$ so that no ``hanging" nodes are necessary in the construction of the finite element functions.

\subsection{Primal WG-FEM}

Given a finite element partition $\T_h$ of the polygonal or polyhedral domain $\Omega$ by $d$-dimensional polytopes (polygons in 2D or polyhedra in 3D) and two integers $k>0$ and $s\ge 0$, we define a finite dimensional space $W_h^{(k,s)}$ as follows:
\begin{equation}\label{wh:01}
W_h^{(k,s)}=\{w=\{w_0,w_b\}:\; w_0|_T\in P_k(T), \ w_b|_e\in P_{s}(e) \ \forall \ e\in\E_h\cap\pT\mbox{ and } T\in\T_h\}.
\end{equation}
For convenience of discussion, we also introduce the following finite element space on $\partial\T_h$ associated with the partition $\E_h$:
\begin{equation}\label{LambdaSpace}
\Lambda_h^{(s)}=\{\sigma:\; \ \sigma|_e\in P_s(e)\; \forall e\in\E_h\}.
\end{equation}

\smallskip

{\color{blue}\bf Weak finite element spaces:}\
Denote by $W_h^0\subset W_h^{(k,s)}$ the closed subspace consisting of functions with vanishing boundary value; i.e.,
\begin{equation}\label{wh:02}
W_h^0=\{w=\{w_0,w_b\}\in W_h^{(k,s)}:\; w_b|_e = 0 \;\ \forall e\in\E_h\cap\partial\Omega\}.
\end{equation}
\smallskip

{\color{blue}\bf Discrete weak gradient:}\ For $w=\{w_0,w_b\}\in W_h^{(k,s)}$, we define on each element $T\in\T_h$ its discrete weak gradient $(\nabla_{w} w)|_T\in [P_{r}(T)]^d$ as follows:
\begin{equation}\label{w-g}
(\nabla_{w}w|_T, \bv)_T = -(w_0,\;\nabla\cdot \bv)_T+ \l w_b,\bv\cdot\bn_T \r_\pT\qquad\forall \bv\in [P_{r}(T)]^d,
\end{equation}
where $\bn_T$ stands for the unit outward normal vector on $\pT$.
\smallskip

{\color{blue}\bf Stabilizer/smoother:}\
The stabilizer/smoother in the current application is given as follows:
\begin{equation}\label{smoother}
s_p(w,\phi)  = \rho \sum_{T\in {\mathcal T}_h}
h_T^{-1}\langle Q_b w_0-w_b,\;Q_b \phi_0-\phi_b
\rangle_{\partial T},\quad w,\;\phi\in W_h^{(k,s)},
\end{equation}
where $Q_b$ is the local $L^2$ projection operator onto the space $\Lambda_h^{(s)}$, and $\rho>0$ is a parameter at user's choice. This form of the stabilizer/smoother is in response to the continuity requirement of the finite element functions in $H^1(\Omega)$. This smoother intends to provide a ``weak continuity'' for the weak finite element functions in the WG context.

\begin{algorithm-primal-wgfem}
The primal weak Galerkin finite element method (WG-FEM) \cite{wg-soe,wg-PkPk-1,wg-systematic} for the model problem \eqref{pde}-\eqref{bc} is to find $u_h=\{u_0,u_b\}\in W_h^0$ such that
\begin{equation}\label{wg}
(a\nabla_wu_h,\nabla_w v)+ s_p(u_h,\;v)=(f,\; v_0)\qquad \forall v=\{v_0,v_b\}\in W_h^0.
\end{equation}
\end{algorithm-primal-wgfem}

\smallskip
A typical example on the value of the integers $s$ (in the construction of $u_b$) and $r$ (in the construction of the discrete weak gradient) is given by
$$
s=k\ \mbox{or}\ k-1,\;\; r=k-1.
$$
The choice of $r=k-1$ for the discrete weak gradient space is due to the understanding of ``the gradient of polynomials of degree $k$ is a vector-valued polynomial of degree $k-1$''. For the case of $s=k$, the $L^2$ projection operator $Q_b$ in the construction of the stabilizer/smoother $s_p(\cdot,\cdot)$ becomes to be the identity operator so that the stabilizer/smoother can be rewritten as follows:
\begin{equation}\label{stabilizer:02}
s_p(w,\phi)  = \rho \sum_{T\in {\mathcal T}_h}
h_T^{-1}\langle w_0-w_b,\; \phi_0-\phi_b
\rangle_{\partial T},\quad w,\;\phi\in W_h^{(k,k)}.
\end{equation}
A systematic study of the primal WG-FEM method for arbitrary combinations of $(k,s)$ and $r$ can be found in \cite{wg-systematic}.

\subsection{Primal-Mixed WG-FEM}

In the primal-mixed weak Galerkin FEM, we may use the weak finite element space $W_h^0$ given by (\ref{wh:02}) to approximate the primal variable $u=u(\bx)$ in (\ref{pm1})-(\ref{pm2}). For any $v\in W_h^0$, the discrete weak gradient $\nabla_w v$ on $T\in\T_h$ is defined by the equation \eqref{w-g}. For a numerical approximation of the flux variable $\bq$, we introduce the following finite element space:
\begin{equation}\label{pm-vector-space}
\bV_h=\left\{\bv\in [L^2(\Omega)]^d:\;\;\bv|_T\in [P_{m}(T)]^d\right\},
\end{equation}
where $m\ge 0$ is an integer.
 \medskip

\begin{algorithm-primalmixed-wgfem}
The primal-mixed weak Galerkin finite element method (Primal-Mixed WG-FEM) for the model problem (\ref{pde})-(\ref{bc}) is to seek $\bq_h\in \bV_h$ and $u_h=\{u_0,u_b\}\in W_h^0$ such that
\begin{eqnarray}
(a^{-1}\bq_h,\bv)+(\bv,\nabla_wu_h)&=&0\;\qquad\forall \bv\in \bV_h,\label{wg-pm1}\\
-s_p(u_h, w)+(\bq_h,\;\nabla_ww)&=&-(f,\;w_0)\qquad \forall w=\{w_0,w_b\}\in W_h^0,\label{wg-pm2}
\end{eqnarray}
where $s_p(u_h, w)$ is given in (\ref{smoother}).
\end{algorithm-primalmixed-wgfem}

\smallskip
A typical example on the value of the integers $s$ (in the construction of $u_b$), $r$ (in the construction of the discrete weak gradient), and $m$ (in the construction of $\bq_h$) is given by
$$
s=k\ \mbox{or}\ k-1,\;\; r=k-1,\; \; m\ge r
$$
so that an appropriate {\em inf-sup} condition holds true. Like the primal WG-FEM, the selection of $r=k-1$ for the discrete weak gradient space is also due to the understanding of ``the gradient of polynomials of degree $k$ is a vector-valued polynomial of degree $k-1$''. The condition of $m\ge r$ ensures a satisfaction of the following {\em inf-sup} condition: $\exists\ \beta>0$ such that
$$
\sup_{\bv\in \bV_h, \bv\neq 0} \frac{(\bv,\nabla_w w)}{\| \bv\|_0} \ge \beta \|\nabla_w w\|_0,\qquad \mbox{for } w\in W_h^0.
$$

\smallskip

To the authors' knowledge, there was no published work available in the existing literature that addresses the theory and convergence of the primal-mixed weak Galerkin finite element method \eqref{wg-pm1}-\eqref{wg-pm2}, as it was believed that the primal-mixed WG-FEM has a direct connection with the hybridizable discontinuous Galerkin (HDG) method for the case of $s=k, \ r=k-1$, and $m=r$. But these two methods are expected to be different for other selections of the finite element spaces $\bV_h$ and $W_h^0$. Interested readers are encouraged to conduct a systematic study on the convergence and stability of the primal-mixed WG-FEM by following the framework developed in \cite{wg-systematic}.

\subsection{Mixed WG-FEM}

The mixed weak Galerkin finite element method is based on the variational form (\ref{dm1})-(\ref{dm2}) for the model second order elliptic problem \eqref{pde}-\eqref{bc}. Note that the principle differential operator in (\ref{dm1})-(\ref{dm2}) is the divergence for which a discrete {\em weak divergence} must be introduced.

{\color{blue}\bf Weak finite element spaces:}\
Given a finite element partition $\T_h$ of the polygonal or polyhedral domain $\Omega$ by $d$-dimensional polytopes (polygons in 2D or polyhedra in 3D) and three integers $k\ge 0$, $s\ge 0$, and $r\ge 0$, we construct two finite dimensional spaces as follows:
\begin{equation}\label{mixed-wgfem-Uspace}
W_h^{(r)}=\{w\in L^2(\Omega):\; w|_T\in P_{r}(T) \; \; \forall T\in\T_h\},
\end{equation}
and
\begin{equation}\label{mixed-wgfem-Qspace}
Z_h^{(k,s)}=\left\{ \bv=\{\bv_0, \bv_b\}:\; \bv_0|_T\in [P_k(T)]^d,
\bv_b|_e\equiv v_b\bn_T, v_b|_e\in P_s(e),\;\forall e\in \E_h\cap\pT, T\in\T_h\right\},
\end{equation}
where $\bn_T$ is the unit outward normal vector on $e\in \E_h\cap\pT$ -- portion of the boundary $\pT$. Note that $v_b$ was meant to represent the outward normal component of the vector-valued weak function $\bv$ on each element $T\in\T_h$.

For convenience, we introduce the following finite element space on $\partial\T_h$ associated with the partition $\E_h$:
\begin{equation}\label{LambdaSpace:new}
\Lambda_h^{(s)}=\left\{\sigma:\;  \sigma|_e\in P_s(e)\; \; \forall e\in\E_h\right\}.
\end{equation}

{\color{blue}\bf Discrete weak divergence:}\ For any $\bv=\{\bv_0,\bv_b\}\in Z_h^{(k,s)}$, on each element $T\in\T_h$ we define its discrete weak divergence $(\nabla_{w}\cdot\bv)|_T\in P_{r}(T)$ as follows:
\begin{equation}\label{d-w-div}
(\nabla_{w}\cdot\bv|_T, \phi)_T = -(\bv_0,\;\nabla \phi)_T+ \l \bv_b,\;
\phi\bn_T \r_\pT\qquad\forall \phi\in P_{r}(T).
\end{equation}

{\color{blue}\bf Stabilizer/smoother:}\ The stabilizer/smoother shall be defined to provide a weak characterization for the continuity of the vector field along the normal direction across each element interface. The commonly used form of the stabilizer/smoother in the Mixed WG-FEM method is given by
\begin{equation}\label{smoother-mixed-wgfem}
s_m(\bq,\bv) = \rho \sum_{T\in {\mathcal T}_h}
h_T^{\alpha}\langle Q_b (\bq_0\cdot\bn_T)-q_b,\;Q_b (\bv_0\cdot\bn_T)-v_b
\rangle_{\pT}
\end{equation}
for $\bv,\;\bq\in Z_h^{(k,s)}$,
where $Q_b$ is the $L^2$ projection operator onto the space $\Lambda_h^{(s)}$, $\rho>0$ and $\alpha$ are parameters at user's choice. This choice of the smoother is in response to the continuity of the vector field across the element interfaces along the normal direction. The stabilization parameters in $\rho h_T^\alpha$ are related to the dimensional matching for all the forms involved.

\begin{algorithm-mixed-wgfem}\ Let $Z_h=Z_h^{(k,s)}$ and $W_h=W_h^{(r)}$.
The mixed weak Galerkin finite element method (Mixed WG-FEM) \cite{wg-mixed} for the model problem (\ref{pde})-(\ref{bc}) is to find $\bq_h=\{\bq_0,\bq_b\}\in Z_h$ and $u_h\in W_h$ such that
\begin{eqnarray}
s_m(\bq_h,\bv)+(a^{-1}\bq_0,\bv_0)-(\nabla_w\cdot\bv,\;u_h)&=&0\qquad \forall \bv\in Z_h,\label{wg-dm1}\\
(\nabla_w\cdot\bq_h,\;w)&=&(f,\;w)\qquad \forall w\in W_h.\label{wg-dm2}
\end{eqnarray}
\end{algorithm-mixed-wgfem}

\begin{remark}
The description of the mixed WG-FEM in \eqref{wg-dm1}-\eqref{wg-dm2} is more general than the one presented in \cite{wg-mixed} in that the finite element functions are allowed to have more options in their construction. For example, the component $q_b$ is set to be in $P_s(e)$ for an independent integer $s$ from $k$ (which is for $\bq_0$), and the space $W_h$ consists of piecewise polynomials of degree $r$ -- another independent integer. The current presentation follows the spirit of the systematic study for the primal WG-FEM method conducted in \cite{wg-systematic}. Interested researchers may fill all the gaps through a study for the numerical scheme \eqref{wg-dm1}-\eqref{wg-dm2}.
\end{remark}

To ensure a satisfaction of the {\em inf-sup} condition for the bilinear form $(\nabla_w\cdot\bv,\;u_h)$,
one may select the non-negative integers $k,s,$ and $r$ so that
\begin{equation}\label{inf-sup-pre}
k\ge r-1,\; s\ge r.
\end{equation}
Under the condition \eqref{inf-sup-pre}, it is not hard to derive the {\em inf-sup} condition of Babu\u{s}ka \cite{babuska} and Brezzi \cite{brezzi} by following a procedure developed by M. Fortin \cite{mixed} based on the $L^2$ norm for the variable $u_h$. In fact, under the condition of \eqref{inf-sup-pre}, it can be shown that the following identity holds true:
$$
(\nabla_w\cdot ({\bf Q_h\bq}), w)_T = (\nabla\cdot\bq, w)_T\qquad \forall w\in P_{r}(T),
$$
where ${\bf Q_h\bq}=\{{\bf Q}_0\bq, Q_b(\bq\cdot\bn_T) \bn_T\}\in Z_h$ is the usual $L^2$ projection of $\bq\in [H^1(\Omega)]^d$ with ${\bf Q}_0$ and $Q_b$ being the local $L^2$ projection operators onto the corresponding local finite element spaces.

It should be noted that the {\em inf-sup} condition of Babu\u{s}ka \cite{babuska} and Brezzi \cite{brezzi} may still hold true even under different norms for the finite element space $W_h$ if the integers do not satisfy \eqref{inf-sup-pre}; see \cite{wg-mixed} for such a development.

The parameter $\alpha$ should be adjusted accordingly for the best possible convergence in the error estimate. A systematic study on the convergence and stability of the mixed WG-FEM remains to be conducted.

\section{Hybridized Mixed WG-FEM}\label{Section:hmwgfem}

Like the hybridized mixed finite element method \eqref{hmfem:01}-\eqref{hmfem:03}, the mixed WG-FEM \eqref{wg-dm1}-\eqref{wg-dm2} can be hybridized by following the idea of Fraeijs de Veubeke \cite{hmf} that relaxes the global nature of the unknown variables corresponding to $q_b\bn_T$ as in $\bq_h=\{\bq_0, q_b\bn_T\}$ on each element $T\in\T_h$. Note that the weak finite element functions $\bv=\{\bv_0, v_b\bn_T\}\in Z_h=Z_h^{(k,s)}$ defined by \eqref{mixed-wgfem-Qspace}
are single-valued along the normal direction on the $d-1$ dimensional finite element partition $\E_h$ for the set of the element boundaries $\partial\T_h$. The idea of Fraeijs de Veubeke hybridization is to eliminate this single-value requirement for $v_b\bn_T$ from the space $Z_h$ thereby obtaining a space $\tilde Z_h$ and to impose instead the desired single-value property on the solution $\bq_h\in \tilde Z_h$ via Lagrangian multipliers.

\smallskip
{\color{blue}\bf Weak finite element spaces:}\
For the given polygonal or polyhedral partition $\T_h$, we introduce the finite element space $\tilde Z_h:=\tilde Z_h^{(k,s)}$ as follows:
\begin{equation}\label{Zh-MWG-hybridized-0}
\tilde Z_h^{(k,s)}=\prod_{T\in\T_h}\tilde Z_{k,s}(T),
\end{equation}
where
\begin{equation}\label{Zh-MWG-hybridized}
\tilde Z_{k,s}(T)=\left\{\bq=\{\bq_0,q_b\bn_T\}:\;  \bq_0\in [P_k(T)]^d,\;  q_b|_e\in P_s(e) \; \forall e\in \E_h\cap\pT\right\}.
\end{equation}
Moreover, denote by $\Lambda_h:=\Lambda_h^{(s)}\subset L^2(\partial\T_h)$ the finite dimensional space given by
\begin{equation}\label{LagrangeMWG-Space}
\Lambda_h^{(s)}=\left\{\sigma: \  \ \sigma|_e \in P_s(e)\; \forall e\in \E_h\right\},
\end{equation}
where $P_s(e)$ stands for the polynomial subspace consisting of all polynomials of degree $s$ and less on the element $e\in\E_h$. Denote by $\Lambda_h^0\subset \Lambda_h$ the subspace of $\Lambda_h$ with vanishing value on $\partial\Omega$; i.e.,
\begin{equation}\label{LagrangeMWG-Space0}
\Lambda_h^0=\left\{\sigma\in\Lambda_h: \;  \sigma|_{e\cap\partial\Omega}=0 \;\; \forall e\in\E_h\right\}.
\end{equation}

{\color{blue}\bf Discrete weak divergence:}\ For any $\bv=\{\bv_0,\bv_b\}\in \tilde Z_{k,s}(T)$, the discrete weak divergence $\nabla_{w}\cdot\bv\in P_{r}(T)$ on $T\in\T_h$ is given by
\begin{equation}\label{d-w-div-hwgfem}
(\nabla_{w}\cdot\bv, \phi)_T = -(\bv_0,\;\nabla \phi)_T+ \l \bv_b,\;
\phi\bn_T \r_\pT\qquad\forall \phi\in P_{r}(T).
\end{equation}

{\color{blue}\bf Stabilizer/smoother:}\ On each element $T\in\T_h$, the local stabilizer/smoother is constructed to provide a weak characterization of the continuity of the vector field along the normal direction across the element interfaces. The following is one such example that serves the purpose well:
\begin{equation}\label{smoother-hmixed-wgfem}
s_T(\bq,\bv) = \rho h_T^{\alpha}\langle Q_b (\bq_0\cdot\bn_T)-q_b,\;Q_b (\bv_0\cdot\bn_T)-v_b
\rangle_{\pT}
\end{equation}
for $\bv,\;\bq\in \tilde Z_{k,s}(T)$, where $Q_b$ is the $L^2$ projection operator onto the space $\Lambda_h^{(s)}$ (which is defined locally on each $L^2(e)$), $\rho>0$ and $\alpha$ are parameters at user's choice.

\begin{algorithm-hmwgfem} Let the finite element spaces be given as follows: $W_h=W_h^{(r)}$ by \eqref{mixed-wgfem-Uspace}, $\tilde{Z}_h=\prod_{T\in\T_h} \tilde Z_{k,s}(T)$ by \eqref{Zh-MWG-hybridized-0}, and $\Lambda_h^0$ by \eqref{LagrangeMWG-Space0}. The hybridized mixed weak Galerkin finite element method seeks $\bq_h\in \tilde{Z}_h$, $u_h\in W_h$, and $u_b\in \Lambda_h^0$ satisfying
\begin{eqnarray}
s_T(\bq_h, \bv)+(a^{-1}\bq_{0}, \bv_0)_T && \nonumber \\
- (\nabla_w\cdot\bv, u_h)_T +\langle u_b, \bv_b\cdot\bn_T\rangle_\pT &=& 0\qquad \forall \bv\in \tilde{Z}_{k,s}(T),\ T\in\T_h,\label{hmwgfem:01}\\
(\nabla_w\cdot\bq_h, w)_T &=& (f, w)_T\qquad \forall w\in W_h, \ T\in\T_h,\label{hmwgfem:02}\\
\sum_{T\in\T_h}\langle \sigma, \bq_b\cdot\bn_T\rangle_\pT &=& 0\qquad \forall \sigma\in \Lambda_h^0.\label{hmwgfem:03}
\end{eqnarray}
\end{algorithm-hmwgfem}

The hybridized mixed weak Galerkin finite element method \eqref{hmwgfem:01}-\eqref{hmwgfem:03} is equivalent to the mixed weak Galerkin finite element method \eqref{wg-dm1}-\eqref{wg-dm2} in the sense that the numerical solutions for $\bq_h$ and $u_h$ resulting from both schemes are identical. In fact, the equation \eqref{hmwgfem:03} implies that the solution $\bq_h=\prod_{T\in\T_h}\{\bq_0,q_b\bn_T\}$ arising from \eqref{hmwgfem:01}-\eqref{hmwgfem:03} is single-valued at each inter-element interface $\pT_1\cap\pT_2$ along the normal direction; i.e.,
$$
(q_b\bn_{T_1})|_{T_1}\cdot\bn_{T_1}=-(q_b\bn_{T_2})|_{T_2}\cdot\bn_{T_2}
$$
on the element interface $\pT_1\cap\pT_2$ should the set has dimension $d-1$. Thus, by restricting the test function $\bv$ to the closed subspace $Z_h\subset \tilde Z_h$, one arrives at the system of equations \eqref{wg-dm1}-\eqref{wg-dm2}. Assuming the satisfaction of the {\em inf-sup} condition of Babu\u{s}ka \cite{babuska} and Brezzi \cite{brezzi} (which can be easily verified), the solution uniqueness for \eqref{wg-dm1}-\eqref{wg-dm2} then implies the equivalence of the mixed WG-FEM and the hybridized mixed WG-FEM.

\section{A Reformulation of the Hybridized Mixed WG-FEM}\label{Section:hmwgfem-re}

The goal of this reformulation for the hybridized mixed WG-FEM is to obtain a formulation that can be easily compared with the HDG method.

Consider the hybridized Mixed WG-FEM given by \eqref{hmwgfem:01}-\eqref{hmwgfem:03}. By choosing $\bv=\{{\mathbf 0},v_b\bn_T\}\in \tilde{Z}_{k,s}(T)$, we arrive at
$$
(\nabla_w\cdot\bv, u_h)_T = -(0, \nabla u_h)_T + \langle v_b, u_h\rangle_\pT=\langle v_b, u_h\rangle_\pT.
$$
From the above equation and the fact that $\bv_b\cdot\bn_T=v_b$ on $\pT$, we have from \eqref{hmwgfem:01} that
\begin{equation}\label{Compare:001}
- \rho h_T^\alpha \langle Q_b (\bq_0\cdot\bn_T) - q_b, v_b \rangle_{\pT} + \langle u_b-u_h, v_b \rangle_\pT = 0\quad \forall v_b\in \Lambda_h,
\end{equation}
which leads to
\begin{equation}\label{Compare:002}
- \rho h_T^\alpha \left(Q_b (\bq_0\cdot\bn_T) - q_b\right) + u_b- Q_b u_h = 0,\qquad\mbox{on } \pT
\end{equation}
or equivalently
\begin{equation}\label{Compare:003}
q_b =  Q_b(\bq_0\cdot\bn) + \rho^{-1}h_T^{-\alpha} (Q_b(u_h|_\pT)-u_b),\qquad\mbox{on } \pT.
\end{equation}

Next, by choosing the test function as $\bv=\{\bv_0, {\mathbf 0}\}\in \tilde{Z}_{k,s}(T)$ (i.e., by setting $\bv_b={\mathbf 0}$) in \eqref{hmwgfem:01} we obtain
$$
 \rho h_T^\alpha \langle Q_b (\bq_0\cdot\bn_T) - q_b, Q_b (\bv_0\cdot\bn_T) \rangle_{\pT}+ (a^{-1}\bq_{0}, \bv_{0})_T +(\bv_0, \nabla u_h)_T=0,
$$
which, after applying the integration by parts, leads to
\begin{equation}\label{Compare:005}
\begin{split}
 \rho h_T^\alpha \langle Q_b (\bq_0\cdot\bn_T)& - q_b, Q_b (\bv_0\cdot\bn_T) \rangle_{\pT} + (a^{-1}\bq_{0}, \bv_{0})_T \\
& -(\nabla\cdot\bv_0, u_h)_T+\langle u_h, \bv_0\cdot\bn_T\rangle_\pT=0.
\end{split}
\end{equation}
Substituting \eqref{Compare:003} into \eqref{Compare:005} yields
\begin{eqnarray}\label{Compare:006}
 (a^{-1}\bq_{0}, \bv_{0})_T -(\nabla\cdot\bv_0, u_h)_T+\langle u_b, \bv_0\cdot\bn_T\rangle_\pT+\langle u_h - Q_b u_h, \bv_0\cdot\bn_T\rangle_\pT =0.
\end{eqnarray}

Third, from the definition of the weak divergence \eqref{d-w-div-hwgfem}, the second equation \eqref{hmwgfem:02} in the mixed WG-FEM can be rewritten as
\begin{equation}\label{Compare:010}
-(\bq_0, \nabla w)_T +\langle \bq_b\cdot\bn_T, w\rangle_\pT = (f,w)_T.
\end{equation}

In summary, a reformulation of the hybridized mixed WG-FEM can be stated as follows:

\noindent{\bf Hybridized Mixed WG-FEM (version-2):}\ {\em
Let the finite element spaces be given as follows: $W_h=W_h^{(r)}$  by \eqref{mixed-wgfem-Uspace}, $\tilde{Z}_h=\prod_{T\in\T_h} \tilde Z_{k,s}(T)$ by \eqref{Zh-MWG-hybridized-0}, and $\Lambda_h^0$ by \eqref{LagrangeMWG-Space0}. The hybridized mixed weak Galerkin finite element method seeks $\bq_h\in \tilde{Z}_h$, $u_h\in W_h$, and $u_b\in \Lambda_h^0$ satisfying, on each element $T\in\T_h$, the following equations:
\begin{eqnarray}
(a^{-1}\bq_{0}, \bv_{0})_T -(\nabla\cdot\bv_0, u_h)_T&&\nonumber\\
+\langle u_b, \bv_0\cdot\bn_T\rangle_\pT+\langle u_h - Q_b u_h, \bv_0\cdot\bn_T\rangle_\pT &=&0\quad \forall\bv_0\in [P_k(T)]^d,\label{version2-hmwgfem:01}\\
-(\bq_0, \nabla w)_T +\langle \bq_b\cdot\bn_T, w\rangle_\pT &=& (f,w)_T
\quad \forall w\in P_r(T),\label{version2-hmwgfem:02}\\
\sum_{T\in\T_h}\langle \sigma, \bq_b\cdot\bn_T\rangle_\pT &=& 0\quad \forall \sigma\in \Lambda_h^0,\label{version2-hmwgfem:03}\\
Q_b(\bq_0\cdot\bn) + \rho^{-1}h_T^{-\alpha} (Q_b(u_h|_\pT)-u_b)&=&q_b,\quad\mbox{on } \pT, \label{version2-hmwgfem:04}
\end{eqnarray}
where $Q_b$ is the $L^2$ projection operator onto the space $\Lambda_h^{(s)}$ (which is locally defined on $L^2(e)$ for each element $e\in\E_h$).
}

\section{HDG is a special case of the Hybridized Mixed WG-FEM}\label{Section:hdg}

The goal of this section is to draw a connection between the hybridized mixed WG-FEM and the HDG method \cite{cgl}. Let us first describe the HDG scheme for the model problem \eqref{pde}-\eqref{bc} by following the presentation in \cite{cock}.

Let $\T_h$ be a finite element partition of the polygonal or polyhedral domain $\Omega$ consisting of
polygons in two dimension or polyhedra in three dimension. Denote by $\E_h$ the set of all edges or flat faces in $\T_h$. Define the set of element boundaries by $\partial\T_h=\bigcup_{T\in\T_h}\pT$. Note that $\E_h$ can be viewed as a finite element partition of the set $\partial\T_h$ which is consistent with the partition $\T_h$.

Introduce the following notations:
\begin{eqnarray*}
  (v,w)_{\partial\T_h}&=&\sum_{T\in\T_h}(v,w)_T=\sum_{T\in\T_h}\int_T vw dT,\\
   \l v,w\r_{\partial\T_h}&=&\sum_{T\in\T_h} \l v,w\r_\pT=\sum_{T\in\T_h} \int_\pT vw ds,
\end{eqnarray*}
where $\int_\pT vw ds$ represents the integral on the element boundary $\pT$ with $ds$ being the differential of the boundary length/area.

There are three unknown variables $(\tilde\bq_h,u_h,\hat{u}_h)\in \bV_h\times W_h\times M_h$ involved in the HDG formulation introduced in 2009 \cite{cgl}, where the associated finite element spaces are defined as follows:
\begin{eqnarray}
\bV_h&=&\{\bv\in {\bf L}^2(\Omega):\; \bv|_T\in \bV(T), \;\forall T\in\T_h\},\label{hdg-space1}\\
W_h&=&\{w\in L^2(\Omega):\; w|_T\in W(T), \;\forall T\in\T_h\},\label{hdg-space2}\\
M_h&=&\{\mu\in L^2(\E_h):\; \mu|_e\in M(e), \;\forall e\in\E_h\},\label{hdg-space3}
\end{eqnarray}
where $\bV(T)$, $W(T)$, and $M(e)$ are local spaces consisting of polynomials of various degrees.

\begin{algorithm-hdg}
The Hybridizable Discontinuous Galerkin (HDG) method \cite{cgl} for the model problem (\ref{pde})-(\ref{bc}) seeks  $(\tilde\bq_h,u_h,\hat{u}_h)\in \bV_h\times W_h\times M_h$ such that, with $c=a^{-1}$,
\begin{eqnarray}
(c\tilde\bq_h,\bv)_{\T_h}-(u_h, \nabla\cdot\bv)_{\T_h}+\langle \hat{u}_h,\bv\cdot\bn\rangle_{\partial\T_h}&=&0,\label{e111}\\
-(\tilde\bq_h, \nabla w)_{\T_h}+\l \hat{\bq}_h\cdot\bn, w\rangle_{\partial\T_h}&=&(f,w),\label{e222}\\
\langle \hat{\bq}_h\cdot\bn, \mu\rangle_{\partial\T_h\setminus\partial\Omega}&=&0,\label{e333}\\
\l \hat{u}_h,\mu\r_{\partial\Omega}&=&0,\label{e444}
\end{eqnarray}
for all $(\bv,w,\mu)\in \bV_h\times W_h\times M_h$ and
\begin{equation}\label{e555}
\hat{\bq}_h\cdot\bn=\tilde\bq_h\cdot\bn+\tau (u_h-\hat{u}_h)\quad {\rm on} \;\pT \; \forall T\in\T_h,
\end{equation}
where $\tau$ is the stabilization function.
\end{algorithm-hdg}

{\color{blue} We ask the question of whether the hybridized mixed WG-FEM is equivalent to the HDG method.} To address this question, we shall make a comparison between the two methods on the following set of defining variables:
\begin{enumerate}
\item The approximation functions and the spaces where they belong to;
\item The equations that define the approximation functions; i.e., the discrete linear systems for each of them.
\end{enumerate}

{\color{blue}\bf On approximating functions:}\ Table \ref{Table:01} shows a comparison of the approximating functions and the spaces where they are defined. The following three set of variables can be identified as being equivalent, as their spaces have the flexibility of being chosen the same:
$$
\bq_0 \leftrightarrow \tilde\bq_h,\quad u_h \leftrightarrow u_h,\quad
u_b \leftrightarrow \hat{u}_h.
$$
Note that the boundary flux $q_b\bn_T \leftrightarrow \hat{\bq}_h$ may have significant differences between these two methods. In the hybridized mixed WG-FEM, the function $q_b$ is a piecewise polynomial of degree $s$ (i.e., the same as in the definition of the Lagrangian space $\Lambda_h$). But the corresponding variable $\hat{\bq}_h\cdot\bn_T$ in the HDG method was given by the equation \eqref{e555} so that its degree is determined by all three spaces of $\bV(T), \ W(T)$, and $M(\pT)$. This difference indicates that the two equations \eqref{e555} and \eqref{version2-hmwgfem:04} are generally not the same, even with the selection of $\tau=\rho^{-1} h_T^{-\alpha}$ on the stabilization parameter.

\begin{table}[h]
\begin{center}
\caption{Comparison of approximating functions and their spaces for the hybridized mixed WG-FEM and the HDG \cite{cgl}.}\label{Table:01}
\begin{tikzpicture}
\clip node (m) [matrix,matrix of nodes,
fill=black!20,inner sep=0pt,
nodes in empty cells,
nodes={minimum height=1cm,minimum width=2.6cm,anchor=center,outer sep=0,font=\sffamily},
row 1/.style={nodes={fill=black,text=white}},
column 1/.style={nodes={fill=gray,text=white,align=right,text width=3.2cm,text depth=0.8ex}},
column 2/.style={text width=5.2cm,align=center,every even row/.style={nodes={fill=white}}},
column 3/.style={text width=5.2cm,align=center,every even row/.style={nodes={fill=white}},},
row 1 column 1/.style={nodes={fill=gray}},
prefix after command={[rounded corners=4mm] (m.north east) rectangle (m.south west)}
] {
                     & Hybridized Mixed WG-FEM   & HDG \\
Vector Field $\bq$:\;  & $(\bq_0, q_b\bn_T)|_T \in \tilde{Z}_{k,s}(T)$ & $(\tilde\bq_h, \hat\bq_h)|_T\in \bV(T)\times X(T)$\\
Scalar Function $u$:\; & $u_h|_T\in P_r(T)$                               & $u_h|_T\in W(T)$ \\
Lagrangian Multiplier:\; & $u_b|_\pT\in \Lambda_s(\pT)$ & $\hat{u}_h|_T\in M(\pT)$\\
Remarks:\; & $\Lambda_s(\pT):=\Lambda_h^{0}|_\pT$ & $X(T)$ depends on $\bV(T),\ W(T), \ M(e)$ as shown in \eqref{e555} \\
};
\end{tikzpicture}
\end{center}
\end{table}

{\color{blue}\bf On the discrete linear systems:}\ Table \ref{Table:02} compares the two systems of linear equations arising from the hybridized mixed WG-FEM and the HDG method. The table shows that the first and the fourth equations are generally not identical, while the other two equations are identical. Observe that the difference between Eqn. \eqref{version2-hmwgfem:01} and Eqn. \eqref{e111} is given by
\begin{equation}\label{1st-eqn-diff}
ErrorEqn_1:=\langle u_h - Q_b u_h, \bv_0\cdot\bn_T\rangle_\pT,
\end{equation}
and, assuming $\tau=\rho^{-1}h_T^{-\alpha}$, the difference between Eqn. \eqref{version2-hmwgfem:04} and Eqn. \eqref{e555} is given by
\begin{equation}\label{4th-eqn-diff}
ErrorEqn_4:=(I-Q_b)(\bq_0\cdot\bn_T) + \tau (I-Q_b) (u_h|_\pT),
\end{equation}
where $I$ is the identity operator.

\begin{table}[h]
\begin{center}
\caption{Comparison of the discrete linear systems for the hybridized mixed WG-FEM and the HDG \cite{cgl}.}\label{Table:02}
\begin{tikzpicture}
\clip node (m) [matrix,matrix of nodes,
fill=black!20,inner sep=0pt,
nodes in empty cells,
nodes={minimum height=1cm,minimum width=2.6cm,anchor=center,outer sep=0,font=\sffamily},
row 1/.style={nodes={fill=black,text=white}},
column 1/.style={nodes={fill=gray,text=white,align=right,text width=2.2cm,text depth=0.5ex}},
column 2/.style={text width=4.0cm,align=center,every even row/.style={nodes={fill=white}}},
column 3/.style={text width=2.5cm,align=center,every even row/.style={nodes={fill=white}},},
column 4/.style={text width=3.0cm,align=center,every even row/.style={nodes={fill=white}},},
row 1 column 1/.style={nodes={fill=gray}},
prefix after command={[rounded corners=4mm] (m.north east) rectangle (m.south west)}
] {
                     & Hybridized Mixed WG-FEM   & HDG & Comparison\\
1st Equation:\;  & Eqn. \eqref{version2-hmwgfem:01} & Eqn. \eqref{e111} & Not identical \\
2nd Equation:\; &  Eqn. \eqref{version2-hmwgfem:02}  & Eqn. \eqref{e222} & Identical \\
3rd Equation:\; & Eqn. \eqref{version2-hmwgfem:03} & Eqn. \eqref{e333} & Identical \\
4th Equation:\; & Eqn. \eqref{version2-hmwgfem:04} & Eqn. \eqref{e555} & Not Identical\\
};
\end{tikzpicture}
\end{center}
\end{table}

{\color{black}\bf Is the HDG \cite{cgl} equivalent to the Hybridized Mixed WG-FEM?}\; {\color{blue} The answer is generally negative: these two methods are generally different from each other, but they might be equivalent for some particular examples of the finite element spaces}. For example, Table \ref{Table:02} shows that the two methods are equivalent as linear systems if both $ErrorEqn_1$ \eqref{1st-eqn-diff} and $ErrorEqn_4$ \eqref{4th-eqn-diff} equal to zero. For $ErrorEqn_4=0$, one must have
\begin{eqnarray}
(I-Q_b)(\bq_0\cdot\bn_T)&=&0\qquad \forall \bq_0 \in [P_k(T)]^d \mbox{ or } \bV(T),\label{Error:001}\\
(I-Q_b) (u_h|_\pT) &=&0\qquad \forall u_h\in P_r(T) \mbox { or } W(T).\label{Error:002}
\end{eqnarray}
Observe that \eqref{Error:002} implies $ErrorEqn_1=0$. Recall that $Q_b$ is the $L^2$ projection operator onto the Lagrangian space ($\Lambda_h$ in hybridized mixed WG-FEM and $M_h$ in HDG --- both should consist of piecewise polynomials on the edge partition $\E_h$ and they can be chosen to be identical). When the Lagrangian space (which consists of piecewise polynomials of degree $s$) is sufficiently rich that covers the trace of $u_h$ (polynomials of degree $r$) and $\bq_0\cdot\bn_T$ (polynomials of degree $k$) or equivalently when $s\ge k$ and $s\ge r$, we would have the validity of \eqref{Error:001} and \eqref{Error:002} so that the two methods are equivalent. Otherwise, at least one of the two equations \eqref{Error:001} and \eqref{Error:002} will not be satisfied so that the HDG and the hybridized mixed WG-FEM are totally different numerical schemes. {\color{blue} In conclusion, the HDG is a special case of the hybridized mixed WG-FEM (the case of sufficiently rich Lagrangian space that covers the trace of both the vector field $\bq_h$ and the scalar field $u_h$), and these two methods are generally different from each other.}

\section{Primal WG $\neq$ the HDG \cite{cgl}}\label{Section:WG-NE-HDG}

Let us first rewrite the HDG method \eqref{e111}-\eqref{e555} \cite{cgl} into a form that is convenient for making comparisons with the weak Galerkin finite element method \eqref{wg}.

For $(w,\mu)\in W_h\times M_h$, one may define $\tilde\bq_{_{w,\mu}}\in \bV_h$ such that
\begin{equation}\label{flux}
(a^{-1}\tilde\bq_{_{w,\mu}},\bv)_{\T_h}-(w, \nabla\cdot\bv)_{\T_h}+\langle \mu,\bv\cdot\bn\rangle_{\partial\T_h}=0 \qquad \forall \bv\in \bV_h.
\end{equation}
By letting $\bv=\tilde\bq_{_{u_h,\hat{u}_h}}$ in (\ref{flux}) one arrives at
\begin{equation}\label{flux1}
(a^{-1}\tilde\bq_{_{w,\mu}},\tilde\bq_{_{u_h,\hat{u}_h}})_{\T_h}-(w, \nabla\cdot\tilde\bq_{_{u_h,\hat{u}_h}})_{\T_h}+\langle \mu,\tilde\bq_{_{u_h,\hat{u}_h}}\cdot\bn\rangle_{\partial\T_h}=0.
\end{equation}
Using the integration by parts and the equation (\ref{e555}), one may rewrite (\ref{e222}) as follows:
\begin{equation}\label{m1}
(\nabla\cdot\tilde\bq_h,w)_{\T_h}+\tau\langle u_h-\hat{u}_h, w\rangle_{\partial\T_h}=(f,w)\qquad \forall w\in W_h.
\end{equation}
Note that for the solution of the HDG scheme \eqref{e111}-\eqref{e555}, one has $\tilde\bq_h=\tilde\bq_{_{u_h,\hat{u}_h}}$ from \eqref{e111}. Thus, combining (\ref{flux1}) with (\ref{m1}) gives (with $c=a^{-1}$)
\begin{equation}\label{m2}
(c\tilde\bq_{_{w,\mu}},\tilde\bq_{_{u_h,\hat{u}_h}})_{\T_h}+\tau\langle u_h-\hat{u}_h, w\rangle_{\partial\T_h}+\langle \tilde\bq_{_{u_h,\hat{u}_h}}\cdot\bn,\mu\rangle_{\partial\T_h}=(f,w)
\end{equation}
for all $w\in W_h$.
Now substituting $\hat{\bq}_h$ of (\ref{e555}) into (\ref{e333}) yields
\[
\langle\tilde\bq_{_{u_h,\hat{u}_h}}\cdot\bn,\mu\rangle_{\partial\T_h/\partial\Omega}=-\tau\langle u_h-\hat{u}_h, \mu\rangle_{\partial\T_h/\partial\Omega}.
\]
Using the equation above, one may rewrite (\ref{m2}) as follows:
\begin{equation*}
(c\tilde\bq_{_{u_h,\hat{u}_h}},\tilde\bq_{_{w,\mu}})_{\T_h}+\tau\langle u_h-\hat{u}_h, w-\mu\rangle_{\partial\T_h}=(f,w)
\end{equation*}
for all $w\in W_h$ and $\mu\in M_h^0$, where
$$
M_h^0=\left\{\mu\in M_h:\; \mu|_{\partial\Omega}=0\right\}.
$$

Therefore, the HDG method (\ref{e111})-(\ref{e555}) can be reformulated as follows:

\begin{algorithm-hdg} (Version-2). \ The Hybridizable Discontinuous Galerkin (HDG) method seeks $u_h\in W_h$ and $\hat{u}_h\in M_h^0$ such that
\begin{equation}\label{rhdg1}
(c\tilde\bq_{_{u_h,\hat{u}_h}},\tilde\bq_{_{w,\mu}})_{\T_h}+\tau\langle u_h-\hat{u}_h, w-\mu\rangle_{\partial\T_h}=(f,w),
\end{equation}
for all $(w,\mu)\in W_h\times M_h^0$. Here $c=a^{-1}$ and
$\tilde\bq_{_{u_h,\hat{u}_h}}$ is defined as in \eqref{flux}.
\end{algorithm-hdg}

\medskip

We are now in a position to show that the primal {\color{blue} WG-FEM and the HDG are two distinct numerical methods}.

\begin{proposition}
The HDG method (\ref{e111})-(\ref{e555}) and the primal WG-FEM method (\ref{wg}) are not equivalent for the general model problem \eqref{pde}-\eqref{bc}.
\end{proposition}

\begin{proof}
First, with the notations used to describe the WG-FEM, the HDG method (\ref{e111})-(\ref{e555}) or its equivalence (\ref{rhdg1}) can be rewritten as follows: Find $\{u_h,\hat{u}_h\}\equiv\{u_0,u_b\}$ such that
\begin{eqnarray}
(c\tilde\bq_{_{u_0,u_b}},\tilde\bq_{_{v_0,v_b}})_{\T_h}+\tau\langle u_0-u_b, v_0-v_b\rangle_{\partial\T_h}&=&(f,v_0)\quad \forall v=\{v_0, v_b\}\in W_h^0. \label{hdg11}
\end{eqnarray}

Next, recall that the primal WG-FEM method \eqref{wg} \cite{wg,wg-soe,wg-PkPk-1} seeks $U_h=\{u_0, u_b\}$ such that
\begin{eqnarray}\label{wg-88}
(a\nabla_w U_h, \nabla_wv)_{\T_h}+\tau\langle Q_b u_0-u_b, Q_b v_0-v_b\rangle_{\pT_h}&=&(f, v_0)\quad \forall v=\{v_0, v_b\}\in W_h^0.
\end{eqnarray}
Here we have assumed that the same stabilizer/smoother parameters and the same finite element spaces for the approximating functions have been chosen for the two methods (note that different selection on the finite element spaces will make a clear difference between these two methods). Please note the differences on the stabilizer/smoother between these two methods: the primal WG-FEM involves the projection operator $Q_b$ while the HDG does not. But we shall ignore this difference by assuming that the finite element spaces are specially chosen so that $Q_b u_0\equiv u_0$; i.e., the trace space for $u_b$ is assumed to be sufficiently rich to cover the trace of the variable $u_0$ on each element boundary.

From the equation (\ref{flux}) and the definition of the weak gradient we have
\begin{eqnarray}
\nabla_w v&=&\bbQ_h (c\tilde\bq_{_{v_0,v_b}}),\quad \mbox{for}\;\; v=\{v_0,v_b\}\in W_h^0,\label{z0}\\
\nabla_w U_h&=&\bbQ_h (c\tilde\bq_{_{u_0,u_b}}),\quad \mbox{for}\;\; U_h=\{u_0,u_b\}\in W_h^0,\label{z1}
\end{eqnarray}
where $\bbQ_h$ is the element-wise $L^2$ projection operator onto $\bV(T)$ for each element $T\in\T_h$. Here we have assumed that $\bV(T)$ (as in HDG) is identical to the gradient space $[P_r(T)]^d$ in WG; see \eqref{w-g} for details.

Using (\ref{z0}) and (\ref{z1}), the WG-FEM system (\ref{wg-88}) can be rewritten as
\begin{eqnarray}
(a\;\bbQ_h (c\tilde\bq_{_{u_h,\hat{u}_h}}), \bbQ_h (c\tilde\bq_{_{w,\mu}}))_{\T_h}+\tau\langle Q_bu_0-u_b, Q_bv_0-v_b\rangle_{\pT_h}&=&(f, v_0)\label{wg11}
\end{eqnarray}
for all $v=\{v_0, v_b\}\in W_h^0$.

\medskip
{\color{blue}\bf  Our justification of $WG \neq HDG$ goes as follows:}
\begin{itemize}
\item The WG reformulation \eqref{wg11} and the HDG reformulation \eqref{hdg11} are based on the same set of basis functions for the same finite element spaces (by choice and for the purpose of making comparisons);
    \item If \eqref{wg11} were equivalent to \eqref{hdg11} (meaning that they have the same solution for any given function $f=f(\bx)$), then one would have (by assuming that $Q_b=I$)
        \begin{equation}\label{ThisIsIt}
        (c\tilde\bq_{_{u_0,u_b}},\tilde\bq_{_{v_0,v_b}})_{\T_h} \equiv (a\;\bbQ_h (c\tilde\bq_{_{u_h,\hat{u}_h}}), \bbQ_h (c\tilde\bq_{_{w,\mu}}))_{\T_h}
        \end{equation}
        for all $\{u_0,u_b\}\in W_h^0$ and $\{v_0,v_b\}\in W_h^0$.
        \item Note that $c=a^{-1}$ in \eqref{ThisIsIt}. Thus, the identify \eqref{ThisIsIt} clearly does not hold true for the case of general variable diffusion coefficient function $a=a(\bx)$ in \eqref{pde}.
\end{itemize}
This completes the justification of the proposition.
\end{proof}

In the rest of this section, we shall present a simple example to show that the identity \eqref{ThisIsIt} does not hold true. To this end, consider the domain $\Omega=(0,1)\times (0,1)$ which is partitioned into rectangular elements $\T_h$, but with only one single element $T=[0,1]\times [0, 1]$. Let the diffusion coefficient be given by $a(\bx)=1+x$. Let the local spaces employed in (\ref{hdg-space1})-(\ref{hdg-space3}) be given by $\bV(T)=[P_0(T)]^2$, $W(T)=P_1(T)$ and $M(e)=P_1(e)$. It follows that $Q_b=I$ and both $\tilde\bq_{_{u_0,u_b}}$ and $\tilde\bq_{_{v_0,v_b}}$ are constant vectors on $T$. Denote by $Q_h$  the $L^2$ projection operator onto $P_{0}(T)$. Since both $\tilde\bq_{_{u_0,u_b}}$ and $\tilde\bq_{_{v_0,v_b}}$ are constant vectors, we then have
\begin{equation}\label{c1}
\bbQ_h (c\tilde\bq_{_{u_0,u_b}})=(Q_h c)\tilde\bq_{_{u_0,u_b}}, \quad \bbQ_h (c\tilde\bq_{_{v_0,v_b}})=(Q_h c)\tilde\bq_{_{v_0,v_b}}.
\end{equation}
It follows from the definition of $Q_h$ that
\[
Q_h(c)=Q_h(a^{-1})=\frac{1}{|T|}\int_T a^{-1} d\bx =\int_0^1\int_0^1 \frac{1}{1+x}dxdy=\ln 2.
\]
Using (\ref{c1}), we have
\begin{eqnarray*}
(a\;\bbQ_h (c\tilde\bq_{_{u_0,u_b}}), \bbQ_h (c\tilde\bq_{_{v_0,v_b}}))_{\T_h}&=&(\tilde\bq_{_{u_0,u_b}}\cdot \tilde\bq_{_{v_0,v_b}})\int _T a Q_h(c)^2d\bx\\
&=&(\ln 2)^2(\tilde\bq_{_{u_0,u_b}}\cdot \tilde\bq_{_{v_0,v_b}})\int_0^1\int_0^1 (1+x)dxdy\\
&=&\frac32(\ln 2)^2 \;(\tilde\bq_{_{u_0,u_b}}\cdot \tilde\bq_{_{v_0,v_b}}).
\end{eqnarray*}
Furthermore, a simple calculation gives the following:
\begin{eqnarray*}
(c\tilde\bq_{_{u_0,u_b}},\tilde\bq_{_{v_0,v_b}})_{\T_h}&=&(\tilde\bq_{_{u_0,u_b}}\cdot \tilde\bq_{_{v_0,v_b}})\int_0^1\int_0^1 \frac{1}{1+x}dxdy=\ln 2\;(\tilde\bq_{_{u_0,u_b}}\cdot \tilde\bq_{_{v_0,v_b}}).
\end{eqnarray*}
Thus, the fact that $\frac32(\ln 2)^2\neq\ln 2$ shows the identity \eqref{ThisIsIt} does not hold true for this simple case.

\section{Remarks on the Second formulation of the HDG \cite{cock,csx}}\label{Section:Remarks}

Other (second) formulations of the HDG method were introduced in 2018 in \cite{cock,csx}. The so-called second formulations of the HDG method are based on four unknown variables for four equations reformulating the model diffusion problem (\ref{pde})-(\ref{bc}).

The following is the second HDG formulation for ($F_a$) stated in \cite{cock}: Find $(\bq_h,\bg_h, u_h,\hat{u}_h)\in \bV_h\times \bV_h\times W_h\times M_h$ such that
\begin{eqnarray}
-(\bg_h,\bv)_{\T_h}-(u_h, \nabla\cdot\bv)_{\T_h}+\langle \hat{u}_h,\bv\cdot\bn\rangle_{\partial\T_h}&=&0\quad\forall\bv\in \bV_h,\label{e11}\\
(\bq_h,\bv)_{\T_h}&=&-(a\bg_h,\bv)_{\T_h} \quad\forall\bv\in \bV_h,\label{e22}\\
-(\bq_h, \nabla w)_{\T_h}+\l \hat{\bq}_h\cdot\bn, w\rangle_{\partial\T_h}&=&(f,w)\quad\forall w\in W_h,\label{e33}\\
\langle \hat{\bq}_h\cdot\bn, \mu\rangle_{\partial\T_h/\partial\Omega}&=&0\quad\forall\mu\in M_h,\label{e44}\\
\l \hat{u}_h,\mu\r_{\partial\Omega}&=&0,\label{e55}
\end{eqnarray}
and
\begin{equation}\label{e66}
\hat{\bq}_h\cdot\bn=\bq_h\cdot\bn+\tau (u_h-\hat{u}_h)\quad\mbox{on } \pT\; \forall T\in\T_h.
\end{equation}

\medskip
We have the following remarks regarding the second HDG method \eqref{e11}-\eqref{e66} and some statements made in \cite{cock}.

\begin{remark}
The second formulation \eqref{e11}-\eqref{e66} was recently developed and published in 2018. It is not difficult to verify that the above second formulation of the HDG method is a reformulation of the primal WG finite element method for the model problem \eqref{pde}-\eqref{bc} with some particular selections on the finite element spaces and the stabilizer/smoother parameters. Although we do not find this formulation appealing (as it involves unnecessary variables), we do appreciate the attentions paid to the development of WG-FEMs by the author of \cite{cock} and do not object any research activities that exploit the connection between existing numerical methods.
\end{remark}

\begin{remark}
Regarding the statement of ``weak Galerkin is a rewriting of HDG'' made in \cite{cock}, we find this statement to be vague and ill-defined. For example, what exactly HDG is as a method? The author of \cite{cock} is encouraged to lay out his principles for the HDG method before making any claims that are scientifically meaningful and accurate. Our understanding of WG-FEM is the following: Weak Galerkin is a generic numerical methodology for PDEs that follows a certain set of basic principles specified in Section \ref{Section:WG}. These principles set a general guideline for an abstract framework in PDE discretization. Consequently, many variations are possible within the weak Galerkin framework in the application to even the same PDE modeling problem. In the authors' view, the weak Galerkin methods have produced new numerical schemes for various PDEs that have not existed in literature. The example illustrated in Section \ref{Section:WG-NE-HDG} shows that the primal WG-FEM is different from the existing HDG algorithms.
\end{remark}

\begin{remark} To the authors' knowledge, the statement of ``The 2013 WG methods [38] are mixed methods'' made in Section 2.2 in \cite{cock} is erroneous. It is in fact trivial to see that the WG-FEM method in \cite{wg} and the mixed finite element method are different for problems with variable diffusion coefficients $a=a(\bx)$.
\end{remark}


\end{document}